\documentclass[11pt]{article}
\usepackage[a4paper,margin=1in]{geometry}
\usepackage[T1]{fontenc}
\usepackage[utf8]{inputenc}
\usepackage{lmodern}
\usepackage{amsmath,amssymb,amsthm,mathtools,bm}
\usepackage{booktabs,array}
\usepackage{microtype}
\usepackage{enumitem}
\usepackage[hidelinks]{hyperref}
\hypersetup{pdftitle={Vorticity-Response Blowup in Axisymmetric Models},pdfauthor={Rishad Shahmurov},pdfsubject={Modified axisymmetric response models and finite-time blowup}}
\usepackage[nameinlink,capitalize]{cleveref}
\newtheorem{theorem}{Theorem}[section]
\newtheorem{lemma}[theorem]{Lemma}
\newtheorem{proposition}[theorem]{Proposition}
\newtheorem{corollary}[theorem]{Corollary}
\theoremstyle{definition}
\newtheorem{remark}[theorem]{Remark}

\newcommand{\R}{\mathbb R}
\newcommand{\T}{\mathbb T}
\newcommand{\A}{\mathcal A}

\newcommand{\ip}[2]{\left\langle #1,#2\right\rangle}

\newcommand{\avg}[1]{\left\langle #1\right\rangle}

\title{\textbf{Vorticity-Response Blowup in Axisymmetric Models}}
\author{Rishad Shahmurov\\
Cellular Products Research and Development, Roswell, Georgia, USA\\
\texttt{rshahmurov@crimson.ua.edu}}
\date{September 2026}

\begin{document}
\maketitle

\begin{abstract}
We study how the orientation and elliptic depth of swirl-generated vorticity response affect singularity formation in axisymmetric model equations.  For a sign-reversed Hou--Li model on the whole line, the dynamics factorize into two real Riccati--diffusion channels.  For every viscosity $\nu>0$, every nontrivial smooth datum in an infinite-dimensional cone of even, nonnegative, monotone channel data develops finite-time amplitude blowup; in the inviscid case, characteristics give the exact blowup time.  For the same smooth swirling inviscid datum, the standard response sign is global whereas the reversed sign blows up.  We then introduce a three-parameter 5D response family separating transport geometry, response orientation, and elliptic depth.  Exact divergence, energy, and scaling identities show that physical three-dimensional incompressibility selects one geometry, weighted 5D incompressibility selects another, and Navier--Stokes scaling makes two elliptic inversions critical.  The frozen physical response is critical and dissipative, while a natural one-inversion 5D-solenoidal model is supercritical and energy-pumping.  These results identify response orientation and causal depth as load-bearing structure in axisymmetric model dynamics.
\end{abstract}

\section{Introduction}

A model equation is most useful when a small structural change produces a large and mathematically interpretable change in dynamics.  Tao's averaged Navier--Stokes equation is a prominent example: the energy cancellation and broad harmonic-analysis bounds are retained, yet a suitable modification of the bilinear interaction admits finite-time blowup \cite{Tao2016}.  The conclusion is not that the modified equation is Navier--Stokes, but that a regularity proof for Navier--Stokes must use structure finer than the properties preserved by the averaging.

Axisymmetric models provide a different setting in which one can ask the same kind of question.  Hou and Li derived a one-dimensional model from the axisymmetric equations near the symmetry axis and identified a strong nonlinear depletion mechanism \cite{HouLi2008}.  Lei and Hou emphasized the stabilizing role of convection in related axisymmetric models \cite{LeiHou2009}; Hou--Liu--Wang studied families in which convection is strengthened \cite{HouLiuWang2018}; and Hou--Wang showed that weakening advection in the Hou--Li model can expose finite-time singularity, including in the viscous problem \cite{HouWang2024}.  Recent whole-line work also emphasizes the dependence of singular and global behavior on the model parameter and boundary normalization \cite{WangLi2026}.  In parallel, the inviscid Hou--Li hierarchy is closely related to generalized two-component Hunter--Saxton systems, for which broad parameter-dependent blowup and global-existence results are known \cite{Wunsch2010,WuWunsch2012,MoonLiu2012,Sarria2015}.

The parameter changed here is different.  We leave the Hou--Li transport, stretching coefficient, viscosity, and Hodge relation unchanged and alter only the orientation of the source by which swirl produces meridional vorticity:
\begin{equation}\label{eq:intro-model}
\begin{aligned}
 u_t+2\psi u_x&=2\psi_xu+\nu u_{xx},\\
 \omega_t+2\psi\omega_x&=\sigma (u^2)_x+\nu\omega_{xx},\\
 -\psi_{xx}&=\omega.
\end{aligned}
\end{equation}
The Hou--Li source sign is $\sigma=1$.  The main result of the one-dimensional part is that every negative sign $\sigma=-c^2<0$ has a robust viscous blowup cone: if
\[
 a_\pm=q\pm cu,\qquad q=\psi_x,
\]
are initially smooth, even, nonnegative, and nonincreasing on $x>0$, and are not both zero, then the corresponding solution blows up in finite time for every $\nu>0$.  The result is therefore not tied to the codimension-one set $q=cu$.  It persists when both real channels are present.

The mechanism is exposed by a factorization.  On the whole line, after fixing the spatial integration constant by decay, one has
\begin{equation}\label{eq:intro-q}
 q_t+2\psi q_x=q^2-\sigma u^2+\nu q_{xx}.
\end{equation}
For $\sigma=c^2>0$, the variable $Z=q+icu$ obeys one complex Riccati--diffusion equation.  For $\sigma=-c^2<0$, the two real variables $a_\pm=q\pm cu$ obey
\begin{equation}\label{eq:intro-split}
 (a_\pm)_t+V(a_\pm)_x=a_\pm^2+\nu(a_\pm)_{xx},
 \qquad V_x=a_++a_-,\qquad V(0,t)=0.
\end{equation}
Inside the positive monotone cone, $V\ge0$ on $x>0$ and $(a_\pm)_x\le0$, so every nonzero channel satisfies
\[
 (a_\pm)_t-\nu(a_\pm)_{xx}\ge a_\pm^2.
\]
A backward-heat-kernel argument then forces finite-time blowup.  The same sign is visible globally in the exact moment identity
\[
 \frac{d}{dt}\int a_\pm^p
 +\nu p(p-1)\int a_\pm^{p-2}(a_\pm)_x^2
 =(p+1)\int a_\pm^{p+1}+\int a_\mp a_\pm^p.
\]
Thus, on this cone, the retained convection reinforces rather than suppresses the positive reaction.

The inviscid system admits an even sharper test.  For the \emph{same} smooth compactly supported swirling data
\[
 q_0=cf,
 \qquad u_0=f,
 \qquad f\ge0,
\]
the standard sign $\sigma=c^2$ is global, while the reversed sign $\sigma=-c^2$ blows up exactly at
\[
 T_*=(2c\|f\|_\infty)^{-1}.
\]
This same-data comparison is not presented as a new classification of the inviscid generalized Hunter--Saxton family; the purpose is to isolate the response orientation without changing domain, transport strength, viscosity coefficient, or initial profile.

The second half of the paper identifies a higher-dimensional analogue.  In the usual 5D axisymmetric variables
\[
 F=\frac{u^\theta}{r},\qquad
 G=\frac{\omega^\theta}{r},\qquad
 -\Delta_5\Phi=G,
\]
the physical source $S=\partial_z(F^2)$ does not determine the meridional velocity instantaneously.  It drives the independent state $G$, which is then inverted through the Hodge operator.  To separate the pieces of that response, we introduce
\begin{equation}\label{eq:intro-5d}
 b_m[\Phi]=(-r\Phi_z,\,m\Phi+r\Phi_r),
 \qquad
 \Phi=-\kappa(-\Delta_5)^{-\beta}\partial_z(F^2).
\end{equation}
The parameter $m$ selects the transport geometry, $\kappa$ the response orientation, and $\beta$ its elliptic depth.  The exact identities
\[
 \operatorname{div}_{3D}b_m=(m-2)\Phi_z,
 \qquad
 \operatorname{div}_{\mu_5}b_m=(m-4)\Phi_z
\]
show that physical 3D incompressibility selects $m=2$, while 5D weighted incompressibility selects $m=4$.  The corresponding scalar energy law is
\begin{equation}\label{eq:intro-energy}
 \frac12\frac{d}{dt}\|F\|_{L^2(d\mu_5)}^2
 +\nu\|\nabla F\|_{L^2(d\mu_5)}^2
 =\frac{m\kappa}{2}
 \|(-\Delta_5)^{-\beta/2}\partial_z(F^2)\|_{L^2(d\mu_5)}^2.
\end{equation}
Navier--Stokes scaling makes $\beta=2$ exactly critical.  The frozen physical source-to-potential response has $(m,\kappa,\beta)=(2,-\nu^{-1},2)$: critical and dissipative in \eqref{eq:intro-energy}.  The one-inversion 5D-solenoidal response has, at the model level, $(m,\kappa,\beta)=(4,1,1)$: it is pumping and supercritical.  The comparison isolates a structural switch in geometry, sign, memory, and operator depth rather than a small coefficient perturbation.

The comparison with Tao's theorem is methodological rather than quantitative.  Tao modifies the full three-dimensional bilinear interaction while preserving its energy cancellation.  Here the setting is a much narrower axisymmetric model hierarchy.  The common message is that coarse invariants are not enough: the orientation and causal depth of the nonlinear response can be load-bearing regularizing structure.

\section{Whole-line normalization and the periodic zero mode}

Set
\[
 q=\psi_x,
 \qquad V=2\psi.
\]
We first fix the precise whole-line problem used throughout the blowup theorem.  We assume that $u$, $q$, and the derivatives required below tend to zero as $|x|\to\infty$, while $\psi$ is allowed to approach finite, possibly different, limits at the two ends.  We normalize $\psi(0,t)=0$.

\begin{proposition}\label{prop:qsystem}
A smooth solution of \eqref{eq:intro-model} with the preceding whole-line normalization satisfies
\begin{align}
 u_t+Vu_x&=2qu+\nu u_{xx},\label{eq:u-q}\\
 q_t+Vq_x&=q^2-\sigma u^2+\nu q_{xx},\label{eq:q-q}\\
 V_x&=2q,\qquad V(0,t)=0.\label{eq:Vq}
\end{align}
Conversely, a smooth decaying solution of \eqref{eq:u-q}--\eqref{eq:Vq}, with $\omega=-q_x$ and $\psi=V/2$, solves \eqref{eq:intro-model}.
\end{proposition}

\begin{proof}
Since $\omega=-\psi_{xx}=-q_x$, the vorticity equation in \eqref{eq:intro-model} becomes
\[
 -q_{xt}-Vq_{xx}=2\sigma uu_x-\nu q_{xxx}.
\]
On the other hand,
\[
 \partial_x(Vq_x-q^2)
 =V_xq_x+Vq_{xx}-2qq_x
 =Vq_{xx},
\]
because $V_x=2q$.  Hence
\[
 \partial_x\bigl(q_t+Vq_x-q^2+\sigma u^2-\nu q_{xx}\bigr)=0.
\]
The expression in parentheses is independent of $x$.  The imposed decay of $u,q,q_x,q_{xx}$ and the boundedness of $V$ force that spatial constant to vanish.  This gives \eqref{eq:q-q}.  The $u$ equation and \eqref{eq:Vq} are immediate.  Differentiating \eqref{eq:q-q} and reversing the preceding calculation proves the converse.
\end{proof}

The zero spatial constant is not automatic on a periodic domain.  This distinction is important because the original Hou--Li global regularity theorem is periodic.

\begin{proposition}\label{prop:periodic}
Let the model be posed on a periodic interval of length $L$, with periodic $\psi$.  Then $\int q=0$ and the integrated $q$ equation is
\begin{equation}\label{eq:periodic-q}
 q_t+Vq_x=q^2-\sigma u^2+\nu q_{xx}+C(t),
\end{equation}
where
\begin{equation}\label{eq:Ct}
 C(t)=\sigma\avg{u^2}-3\avg{q^2},
 \qquad
 \avg{h}:=\frac1L\int_0^L h(x)\,dx.
\end{equation}
\end{proposition}

\begin{proof}
The same differentiation argument as in \cref{prop:qsystem} gives \eqref{eq:periodic-q} with an undetermined $C(t)$.  Integrating over one period and using periodicity,
\[
 \int_0^L Vq_x\,dx=-\int_0^L V_xq\,dx=-2\int_0^Lq^2\,dx.
\]
Since $q=\psi_x$ has zero mean, integration of \eqref{eq:periodic-q} yields
\[
 -2\int q^2=\int q^2-\sigma\int u^2+LC(t),
\]
which is exactly \eqref{eq:Ct}.
\end{proof}

\begin{remark}\label{rem:domain}
The term $C(t)$ is load-bearing.  In the negative-response variables introduced below it appears as the same additive forcing in both Riccati channels.  In particular, the whole-line invariant cone theorem cannot be transferred to the periodic Hou--Li model by simply replacing $\R$ with $\T$.  All blowup statements in \cref{sec:viscous} are whole-line statements under the normalization of \cref{prop:qsystem}.
\end{remark}

\section{Exact inviscid response geometry}

The sign change is already visible at the algebraic level.

\begin{theorem}\label{thm:factor}
Let $(u,q,V)$ solve \eqref{eq:u-q}--\eqref{eq:Vq}.

If $\sigma=c^2>0$ and
\[
 Z=q+icu,
\]
then
\begin{equation}\label{eq:complex}
 Z_t+VZ_x=Z^2+\nu Z_{xx},
 \qquad V_x=2\operatorname{Re}Z.
\end{equation}
If $\sigma=-c^2<0$ and
\[
 a_+=q+cu,
 \qquad a_-=q-cu,
\]
then
\begin{equation}\label{eq:channels}
 (a_\pm)_t+V(a_\pm)_x=a_\pm^2+\nu(a_\pm)_{xx},
 \qquad V_x=a_++a_-.
\end{equation}
\end{theorem}

\begin{proof}
For $\sigma=c^2$, combine \eqref{eq:u-q} and \eqref{eq:q-q}:
\begin{align*}
 (q+icu)_t+V(q+icu)_x
 &=q^2-c^2u^2+i(2cqu)+\nu(q+icu)_{xx}\\
 &=(q+icu)^2+\nu(q+icu)_{xx}.
\end{align*}
The identity $V_x=2q=2\operatorname{Re}Z$ follows from \eqref{eq:Vq}.  If $\sigma=-c^2$, then
\begin{align*}
 (q\pm cu)_t+V(q\pm cu)_x
 &=q^2+c^2u^2\pm2cqu+\nu(q\pm cu)_{xx}\\
 &=(q\pm cu)^2+\nu(q\pm cu)_{xx},
\end{align*}
and $V_x=2q=a_++a_-$.
\end{proof}

For $\nu=0$, both cases can be integrated exactly along characteristics.  We record the formulas because they make the same-data sign comparison transparent.

\begin{theorem}\label{thm:lagrangian}
Assume $\nu=0$ and let $X(\xi,t)$ solve
\[
 X_t(\xi,t)=V(X(\xi,t),t),
 \qquad X(\xi,0)=\xi.
\]
For the positive sign $\sigma=c^2$,
\begin{equation}\label{eq:lag-pos}
 Z(X(\xi,t),t)=\frac{Z_0(\xi)}{1-tZ_0(\xi)},
 \qquad
 X_\xi(\xi,t)=\frac{1}{|1-tZ_0(\xi)|^2}.
\end{equation}
For the negative sign $\sigma=-c^2$,
\begin{equation}\label{eq:lag-neg}
 a_\pm(X(\xi,t),t)=\frac{a_{\pm,0}(\xi)}{1-ta_{\pm,0}(\xi)},
\end{equation}
\begin{equation}\label{eq:jac-neg}
 X_\xi(\xi,t)
 =\frac{1}{(1-ta_{+,0}(\xi))(1-ta_{-,0}(\xi))}.
\end{equation}
The formulas remain classical as long as the displayed denominators stay nonzero and $X(\cdot,t)$ remains an increasing diffeomorphism.
\end{theorem}

\begin{proof}
Along a characteristic, \eqref{eq:complex} with $\nu=0$ gives $\dot Z=Z^2$, hence the first formula in \eqref{eq:lag-pos}.  The Jacobian $J=X_\xi$ obeys
\[
 \dot J=V_x(X,t)J=2\operatorname{Re}Z(X,t)J.
\]
Since
\[
 \frac{d}{dt}\log|1-tZ_0|^2
 =-2\operatorname{Re}\frac{Z_0}{1-tZ_0},
\]
and $J(\xi,0)=1$, integration gives the second formula in \eqref{eq:lag-pos}.

For the negative sign, \eqref{eq:channels} gives independently
\[
 \frac{d}{dt}a_\pm(X,t)=a_\pm(X,t)^2,
\]
which proves \eqref{eq:lag-neg}.  Moreover
\[
 \dot J=V_xJ=(a_++a_-)J.
\]
Substitution of \eqref{eq:lag-neg} and integration yield \eqref{eq:jac-neg}.
\end{proof}

\begin{corollary}\label{cor:samedata}
Let $f\in C_c^\infty(\R)$ be nonnegative and nonzero and let $c>0$.  Use the same initial data
\begin{equation}\label{eq:samedata}
 q_0=cf,
 \qquad u_0=f
\end{equation}
in the two inviscid models.

For $\sigma=c^2$, the corresponding solution is classical for every finite $t\ge0$.  For $\sigma=-c^2$, it blows up at
\begin{equation}\label{eq:samedataT}
 T_*=(2c\|f\|_\infty)^{-1}.
\end{equation}
\end{corollary}

\begin{proof}
For $\sigma=c^2$,
\[
 Z_0=cf(1+i).
\]
Write $s=ctf(\xi)\ge0$.  Then
\[
 |1-tZ_0|^2=|1-s(1+i)|^2=(1-s)^2+s^2
 =2\left(s-\frac12\right)^2+\frac12\ge\frac12.
\]
Thus neither $Z$ nor the Jacobian in \eqref{eq:lag-pos} can develop a finite-time pole.  Outside the compact support of $f$ one has $Z_0=0$ and therefore $X_\xi=1$.  Hence $X(\cdot,t)$ remains an increasing map onto $\R$ for every finite $t$.  The formulas in \cref{thm:lagrangian} consequently give a global classical solution.

For $\sigma=-c^2$,
\[
 a_{+,0}=2cf,
 \qquad a_{-,0}=0.
\]
The first denominator in \eqref{eq:lag-neg} vanishes precisely when $t=(2c\max f)^{-1}$, while the second is identically one.  This proves \eqref{eq:samedataT}.
\end{proof}

\begin{remark}
The inviscid sign-parametric family belongs, after standard changes of variables and normalization, to the generalized two-component Hunter--Saxton hierarchy.  The literature contains much broader periodic and whole-line blowup/global classifications \cite{Wunsch2010,WuWunsch2012,MoonLiu2012,Sarria2015}.  The role of \cref{cor:samedata} is therefore not a priority claim about inviscid classification.  It is a controlled sign experiment: the domain, velocity law, and initial profile are kept fixed and only the vorticity-response orientation is changed.
\end{remark}

\section{Viscous blowup on an infinite-dimensional cone}\label{sec:viscous}

We now fix $\sigma=-c^2<0$ and work with the real channels \eqref{eq:channels}.  It is convenient to regard them as the primary unknowns:
\begin{equation}\label{eq:twochannel}
\begin{aligned}
 (a_+)_t+V(a_+)_x&=a_+^2+\nu(a_+)_{xx},\\
 (a_-)_t+V(a_-)_x&=a_-^2+\nu(a_-)_{xx},\\
 V_x&=a_++a_-,\qquad V(0,t)=0.
\end{aligned}
\end{equation}
The original variables are recovered by
\begin{equation}\label{eq:recoverqu}
 q=\frac{a_++a_-}{2},
 \qquad
 u=\frac{a_+-a_-}{2c},
 \qquad
 \psi=\frac V2,
 \qquad
 \omega=-q_x.
\end{equation}

For smooth rapidly decaying data, standard one-dimensional parabolic theory gives a unique maximal classical solution when $\nu>0$.  A continuation argument adapted to the cone is included in \cref{app:continuation}; in particular, a finite maximal time in the cone can occur only if $\|a_+\|_\infty+\|a_-\|_\infty$ becomes unbounded.  When $\nu=0$, local existence follows from the characteristic equations in \cref{thm:lagrangian}.

\begin{lemma}\label{lem:cone}
Assume that $a_{+,0},a_{-,0}\in C_c^\infty(\R)$ are even, nonnegative, and nonincreasing on $(0,\infty)$.  Then, on the maximal classical lifespan, each $a_\pm$ remains even and nonnegative and satisfies
\begin{equation}\label{eq:monotone}
 (a_\pm)_x(x,t)\le0\qquad(x>0).
\end{equation}
Consequently
\begin{equation}\label{eq:Vsign}
 V(x,t)\ge0\qquad(x>0)
\end{equation}
and
\begin{equation}\label{eq:supersolution}
 (a_\pm)_t-\nu(a_\pm)_{xx}
 =a_\pm^2-V(a_\pm)_x
 \ge a_\pm^2.
\end{equation}
\end{lemma}

\begin{proof}
The system is invariant under $x\mapsto-x$ if $a_\pm$ are even and $V$ is odd.  Since $V$ is recovered uniquely from $V_x=a_++a_-$ and $V(0,t)=0$, uniqueness preserves this parity.

Nonnegativity follows from the scalar maximum principle applied separately to each channel.  Indeed, if a channel were to cross from nonnegative to negative values, at the first spacetime point where its minimum reached zero one would have $(a_\pm)_x=0$, $(a_\pm)_{xx}\ge0$, and $a_\pm^2=0$, contradicting a negative time derivative there.  The same argument along characteristics applies when $\nu=0$.

Set $h_\pm=(a_\pm)_x$ on $x>0$.  Differentiating \eqref{eq:twochannel} and using $V_x=a_++a_-$ gives
\begin{equation}\label{eq:hpm}
 (h_\pm)_t+V(h_\pm)_x
 =\nu(h_\pm)_{xx}+(a_\pm-a_\mp)h_\pm.
\end{equation}
Evenness gives $h_\pm(0,t)=0$, and decay gives $h_\pm(x,t)\to0$ as $x\to\infty$.  The coefficient $a_\pm-a_\mp$ is bounded on every compact subinterval of the classical lifespan.  The maximum principle for the linear equation \eqref{eq:hpm} therefore preserves $h_\pm\le0$ on the half-line.  Since $V(0,t)=0$ and $V_x=a_++a_-\ge0$, \eqref{eq:Vsign} follows.  Combining \eqref{eq:Vsign} and \eqref{eq:monotone}, and using parity on $x<0$, yields \eqref{eq:supersolution}.
\end{proof}

The comparison equation is the classical Fujita problem
\begin{equation}\label{eq:fujita}
 g_t=\nu g_{xx}+g^2,
 \qquad g(0)=g_0\ge0.
\end{equation}
We include the exact criterion needed below.

\begin{lemma}\label{lem:fujita}
Let $\nu>0$ and let $g_0\in L^1(\R)\cap L^\infty(\R)$ be nonnegative and nonzero.  Set
\[
 K_T(x)=\frac{1}{\sqrt{4\pi\nu T}}e^{-x^2/(4\nu T)}.
\]
If
\begin{equation}\label{eq:fujita-criterion}
 T\int_\R g_0(x)K_T(x)\,dx>1,
\end{equation}
then the classical solution of \eqref{eq:fujita} cannot exist up to time $T$.  Such a $T$ exists for every nonzero $g_0$.
\end{lemma}

\begin{proof}
Fix $T$ and define the backward heat kernel
\[
 K(x,t)=\frac{1}{\sqrt{4\pi\nu(T-t)}}
 e^{-x^2/(4\nu(T-t))},
 \qquad 0\le t<T.
\]
Then $K_t+\nu K_{xx}=0$ and $\int_\R K(x,t)\,dx=1$.  For
\[
 Y(t)=\int_\R g(x,t)K(x,t)\,dx,
\]
integration by parts gives
\[
 Y'(t)=\int_\R g(x,t)^2K(x,t)\,dx.
\]
Because $K\,dx$ is a probability measure, Jensen's inequality yields
\[
 Y'(t)\ge Y(t)^2.
\]
Thus
\[
 \frac1{Y(t)}\le\frac1{Y(0)}-t
\]
as long as the solution remains classical.  If $TY(0)>1$, the right-hand side vanishes before $T$, proving the first assertion.

To prove existence of an admissible $T$, write $M=\int g_0>0$.  Dominated convergence gives
\[
 \sqrt{T}\int g_0(x)K_T(x)\,dx
 \longrightarrow\frac{M}{\sqrt{4\pi\nu}}
 \qquad(T\to\infty).
\]
Hence $T\int g_0K_T$ grows like a positive constant times $\sqrt{T}$ and eventually exceeds one.
\end{proof}

\begin{corollary}\label{cor:explicit-upper}
Suppose $g_0\ge0$ and
\[
 M_R:=\int_{-R}^R g_0(x)\,dx>0.
\]
Then every
\begin{equation}\label{eq:explicit-upper}
 T>\max\left\{\frac{R^2}{4\nu},\frac{4\pi e^2\nu}{M_R^2}\right\}
\end{equation}
satisfies \eqref{eq:fujita-criterion}.
\end{corollary}

\begin{proof}
If $T\ge R^2/(4\nu)$, then $e^{-x^2/(4\nu T)}\ge e^{-1}$ on $[-R,R]$.  Therefore
\[
 T\int g_0K_T
 \ge \frac{e^{-1}M_R\sqrt{T}}{\sqrt{4\pi\nu}},
\]
which is greater than one under the second inequality in \eqref{eq:explicit-upper}.
\end{proof}

We can now state the main theorem.

\begin{theorem}\label{thm:main}
Let $c>0$, $\sigma=-c^2$, and $\nu\ge0$.  Assume
\[
 a_{+,0},a_{-,0}\in C_c^\infty(\R)
\]
are even, nonnegative, nonincreasing on $(0,\infty)$, and not both identically zero.  Define $(q_0,u_0)$ by \eqref{eq:recoverqu}.  Then the maximal classical whole-line solution of \eqref{eq:intro-model} blows up in finite time.  More precisely,
\begin{equation}\label{eq:amplitude-blowup}
 \lim_{t\uparrow T_*}
 \bigl(\|a_+(t)\|_\infty+\|a_-(t)\|_\infty\bigr)=\infty.
\end{equation}

If $\nu=0$, then
\begin{equation}\label{eq:exactT}
 T_*=\frac1{A_0},
 \qquad
 A_0:=\max\{\|a_{+,0}\|_\infty,\|a_{-,0}\|_\infty\}.
\end{equation}
If $\nu>0$, then
\begin{equation}\label{eq:lowerlife}
 T_*\ge\frac1{A_0},
\end{equation}
and, for either nonzero channel $a_{j,0}$, every $T$ satisfying
\begin{equation}\label{eq:upperlife}
 T\int_\R a_{j,0}(x)K_T(x)\,dx>1
\end{equation}
satisfies $T_*<T$.
\end{theorem}

\begin{proof}
By \cref{thm:factor}, the negative-response model is exactly the two-channel system \eqref{eq:twochannel}.  By \cref{lem:cone}, the cone assumptions persist throughout the classical lifespan.

Assume first $\nu>0$.  Choose a channel $a_j$ whose initial datum is not identically zero, and let $g$ solve the Fujita problem \eqref{eq:fujita} with $g_0=a_{j,0}$.  Equation \eqref{eq:supersolution} and the parabolic comparison principle give
\[
 a_j(x,t)\ge g(x,t)
\]
as long as both solutions are classical.  By \cref{lem:fujita}, $g$ blows up before every $T$ satisfying \eqref{eq:upperlife}.  Therefore the two-channel solution cannot remain classical up to such a $T$, and in particular $T_*<\infty$.

We next prove the lower bound and identify the quantity that must diverge.  Since each channel is even and nonincreasing on $x>0$,
\[
 A_j(t):=a_j(0,t)=\|a_j(t)\|_\infty.
\]
At $x=0$ one has $V=0$, $(a_j)_x=0$, and $(a_j)_{xx}\le0$.  Hence
\[
 D^+A_j(t)\le A_j(t)^2.
\]
Comparison with the scalar Riccati equation gives
\begin{equation}\label{eq:center-upper}
 A_j(t)\le\frac{A_j(0)}{1-tA_j(0)}
 \qquad\text{for }t<A_j(0)^{-1},
\end{equation}
with the convention that the zero channel remains zero.  Thus both amplitudes stay bounded on every compact subinterval of $[0,A_0^{-1})$.  The continuation criterion proved in \cref{prop:continue} implies \eqref{eq:lowerlife}.  The same criterion shows that a finite $T_*$ can occur only through \eqref{eq:amplitude-blowup}.

Now suppose $\nu=0$.  The exact formulas \eqref{eq:lag-neg}--\eqref{eq:jac-neg} apply.  Since both initial channels are nonnegative, all denominators remain positive for
\[
 0\le t<A_0^{-1},
\]
and the Jacobian is strictly positive.  Outside the compact support of the data, both initial channels vanish and $X_\xi=1$, so the characteristic map remains onto $\R$.  At $t=A_0^{-1}$ at least one denominator vanishes at a point where the corresponding channel attains its maximum.  This proves \eqref{eq:exactT} and \eqref{eq:amplitude-blowup} in the inviscid case.
\end{proof}

The cone is visibly infinite-dimensional in the original variables: it consists of smooth data satisfying
\[
 q_0\pm cu_0\ge0,
\]
with both combinations even and nonincreasing away from the origin.  In particular, $q_0=cu_0$ is only one boundary face of the cone.

\begin{proposition}\label{prop:Lp}
Let a smooth solution of \eqref{eq:twochannel} remain in the cone of \cref{lem:cone}.  For every $p\ge1$ and $j\in\{+,-\}$, with $k$ denoting the other sign,
\begin{align}
 \frac{d}{dt}\int_\R a_j^p\,dx
 &+\nu p(p-1)\int_\R a_j^{p-2}(a_j)_x^2\,dx\nonumber\\
 &=(p+1)\int_\R a_j^{p+1}\,dx
 +\int_\R a_k a_j^p\,dx.\label{eq:Lp}
\end{align}
In particular,
\begin{equation}\label{eq:mass-id}
 \frac{d}{dt}\int a_j
 =2\int a_j^2+\int a_ja_k.
\end{equation}
\end{proposition}

\begin{proof}
Multiply the $j$th equation in \eqref{eq:twochannel} by $pa_j^{p-1}$ and integrate.  The transport term is
\[
 p\int Va_j^{p-1}(a_j)_x
 =\int V(a_j^p)_x
 =-\int V_xa_j^p
 =-\int a_j^{p+1}-\int a_ka_j^p.
\]
The diffusion term is
\[
 p\nu\int a_j^{p-1}(a_j)_{xx}
 =-\nu p(p-1)\int a_j^{p-2}(a_j)_x^2,
\]
and the reaction contributes $p\int a_j^{p+1}$.  Rearrangement gives \eqref{eq:Lp}; setting $p=1$ gives \eqref{eq:mass-id}.
\end{proof}

\begin{remark}
For the semilinear heat equation alone, the coefficient of $\int a_j^{p+1}$ in the corresponding identity is $p$.  In \eqref{eq:Lp}, the retained nonlocal transport raises it to $p+1$ and adds the nonnegative cross term $\int a_ka_j^p$.  This is an exact integral form of the pointwise inequality \eqref{eq:supersolution}.
\end{remark}

\section{A three-parameter 5D response family}

We now isolate the higher-dimensional structure suggested by the response-sign calculation.  In axisymmetric cylindrical variables define
\[
 F=\frac{u^\theta}{r},
 \qquad
 G=\frac{\omega^\theta}{r},
 \qquad
 \Phi=\frac{\psi^\theta}{r},
\]
and
\[
 \Delta_5=\partial_{rr}+\frac3r\partial_r+\partial_{zz},
 \qquad
 d\mu_5=r^3\,dr\,dz,
 \qquad
 \A=-\Delta_5.
\]
For the physical axisymmetric Navier--Stokes equations one has \cite{HouLi2008,HouWang2024}; related 5D-lifted model equations and their dissipation balance have also been studied in \cite{TaoWu2012}
\begin{align}
 F_t+b\cdot\nabla F&=\nu\Delta_5F+2\Phi_zF,\label{eq:physicalF}\\
 G_t+b\cdot\nabla G&=\nu\Delta_5G+\partial_z(F^2),\label{eq:physicalG}\\
 \A\Phi&=G,\label{eq:physicalPhi}
\end{align}
with
\begin{equation}\label{eq:physicalb}
 u^r=-r\Phi_z,
 \qquad
 u^z=2\Phi+r\Phi_r.
\end{equation}
The source $S=\partial_z(F^2)$ therefore passes through an independent parabolic state $G$ before reaching the velocity.

To separate transport geometry from source response, let
\begin{equation}\label{eq:bm}
 b_m[\Phi]=(-r\Phi_z,\,m\Phi+r\Phi_r),
 \qquad m\in\R.
\end{equation}
The radial stretching factor is independent of $m$:
\begin{equation}\label{eq:U}
 U:=\frac{u^r}{r}=-\Phi_z.
\end{equation}

\begin{proposition}\label{prop:divergences}
For every smooth $\Phi$,
\begin{equation}\label{eq:div3}
 \operatorname{div}_{3D}b_m=(m-2)\Phi_z,
\end{equation}
and
\begin{equation}\label{eq:div5m}
 \operatorname{div}_{\mu_5}b_m
 :=\partial_ru^r+\frac3r u^r+\partial_zu^z
 =(m-4)\Phi_z.
\end{equation}
Consequently, $m=2$ is the physical three-dimensional incompressible recovery, while $m=4$ is the weighted 5D-solenoidal recovery.
\end{proposition}

\begin{proof}
From \eqref{eq:bm},
\[
 \partial_ru^r=-\Phi_z-r\Phi_{rz},
 \qquad
 \partial_zu^z=m\Phi_z+r\Phi_{rz}.
\]
Adding $u^r/r=-\Phi_z$ gives
\[
 \partial_ru^r+\frac1r u^r+\partial_zu^z
 =(m-2)\Phi_z,
\]
which is \eqref{eq:div3}.  Adding instead $3u^r/r=-3\Phi_z$ gives \eqref{eq:div5m}.
\end{proof}

We now replace the causal subsystem \eqref{eq:physicalG}--\eqref{eq:physicalPhi} by an instantaneous response law
\begin{equation}\label{eq:response}
 \Phi=-\kappa\A^{-\beta}S,
 \qquad
 S=\partial_z(F^2),
 \qquad \beta>0,
\end{equation}
and evolve $F$ by
\begin{equation}\label{eq:Fmodel}
 F_t+b_m[\Phi]\cdot\nabla F
 =\nu\Delta_5F+2\Phi_zF.
\end{equation}
This produces the three structural parameters
\[
 m\quad\text{(transport geometry)},
 \qquad
 \kappa\quad\text{(response orientation)},
 \qquad
 \beta\quad\text{(elliptic depth)}.
\]

\begin{theorem}\label{thm:5denergy}
Let $(F,\Phi)$ be a sufficiently smooth decaying solution of \eqref{eq:response}--\eqref{eq:Fmodel}.  Then
\begin{equation}\label{eq:5denergy}
 \frac12\frac{d}{dt}\|F\|_{L^2(d\mu_5)}^2
 +\nu\|\nabla F\|_{L^2(d\mu_5)}^2
 =\frac{m\kappa}{2}
 \|\A^{-\beta/2}S\|_{L^2(d\mu_5)}^2.
\end{equation}
\end{theorem}

\begin{proof}
Write
\[
 I:=\int \Phi_zF^2\,d\mu_5.
\]
Multiplying \eqref{eq:Fmodel} by $F$ and integrating gives
\begin{equation}\label{eq:preenergy}
 \frac12\frac{d}{dt}\|F\|_2^2
 +\int F b_m\cdot\nabla F\,d\mu_5
 =-\nu\|\nabla F\|_2^2+2I.
\end{equation}
By \eqref{eq:div5m},
\[
 \int F b_m\cdot\nabla F\,d\mu_5
 =\frac12\int b_m\cdot\nabla(F^2)\,d\mu_5
 =-\frac{m-4}{2}I.
\]
Substitution in \eqref{eq:preenergy} yields
\begin{equation}\label{eq:mI}
 \frac12\frac{d}{dt}\|F\|_2^2+\nu\|\nabla F\|_2^2
 =\frac m2 I.
\end{equation}
Since $S=\partial_z(F^2)$,
\[
 I=\int\Phi_zF^2\,d\mu_5
 =-\int\Phi S\,d\mu_5.
\]
Using \eqref{eq:response} and the positive self-adjointness of $\A$,
\[
 -\ip{\Phi}{S}
 =\kappa\ip{\A^{-\beta}S}{S}
 =\kappa\|\A^{-\beta/2}S\|_2^2.
\]
Combining this identity with \eqref{eq:mI} proves \eqref{eq:5denergy}.
\end{proof}

\begin{proposition}\label{prop:scaling}
Under Navier--Stokes scaling
\begin{equation}\label{eq:Fscale}
 F_\lambda(r,z,t)=\lambda^2F(\lambda r,\lambda z,\lambda^2t),
\end{equation}
the viscous term in \eqref{eq:5denergy} scales like $\lambda$, while the response term scales like
\begin{equation}\label{eq:response-scale}
 \lambda^{5-2\beta}.
\end{equation}
Hence the instantaneous response is scale-critical relative to diffusion exactly when
\begin{equation}\label{eq:beta2}
 \beta=2.
\end{equation}
It is supercritical for $\beta<2$ and subcritical for $\beta>2$.
\end{proposition}

\begin{proof}
The measure $d\mu_5$ scales as $\lambda^{-5}$.  Since $F_\lambda$ has amplitude $\lambda^2$, $\nabla F_\lambda$ has amplitude $\lambda^3$ and therefore
\[
 \|\nabla F_\lambda\|_2^2
 \sim\lambda^{6-5}=\lambda.
\]
The source $S=\partial_z(F^2)$ has amplitude $\lambda^5$.  Applying $\A^{-\beta/2}$ lowers the scaling exponent by $\beta$, so $\A^{-\beta/2}S$ has amplitude $\lambda^{5-\beta}$.  Its squared $L^2(d\mu_5)$ norm scales as
\[
 \lambda^{2(5-\beta)-5}=\lambda^{5-2\beta}.
\]
Equating this exponent with the viscous exponent $1$ gives \eqref{eq:beta2}.
\end{proof}

The physical response is dynamic rather than instantaneous, but its zero-temporal-frequency approximation is informative.  If transport and $G_t$ are frozen in \eqref{eq:physicalG}, then
\[
 \nu\A G\approx S,
 \qquad
 \Phi=\A^{-1}G\approx\nu^{-1}\A^{-2}S.
\]
In the convention \eqref{eq:response}, this is
\begin{equation}\label{eq:frozenphysical}
 (m,\kappa,\beta)=(2,-\nu^{-1},2).
\end{equation}
Thus the frozen physical response is critical and contributes with the dissipative sign in \eqref{eq:5denergy}.  This is a comparison of response orientation and depth, not a replacement of the physical evolution by an elliptic law.

The three parameters are summarized by
\begin{center}
\begin{tabular}{lccccl}
\toprule
response & $m$ & $\kappa$ & $\beta$ & energy sign & scaling \\
\midrule
frozen physical & $2$ & $-\nu^{-1}$ & $2$ & damping & critical \\
5D-solenoidal scalar closure & $4$ & $1$ & $1$ & pumping & supercritical \\
physical geometry, pumping response & $2$ & $1$ & $2$ & pumping & critical \\
\bottomrule
\end{tabular}
\end{center}
The table is not a phase diagram for the full Navier--Stokes equations.  It is an exact accounting identity for the model family \eqref{eq:response}--\eqref{eq:Fmodel}.

\section{A 5D-solenoidal one-inversion model}

We next isolate the pumping, one-inversion member of the response family as a modified equation in its own right.  Let
\begin{equation}\label{eq:oldW}
 \Delta_5W=\partial_z(F^2)
\end{equation}
and define
\begin{equation}\label{eq:b4}
 b_4[W]=(-rW_z,\,4W+rW_r).
\end{equation}
By \cref{prop:divergences}, $b_4$ is exactly divergence-free with respect to $d\mu_5$.  The scalar evolution is
\begin{equation}\label{eq:oldmodel}
 F_t+b_4[W]\cdot\nabla F
 =\nu\Delta_5F+2W_zF.
\end{equation}
Because \eqref{eq:oldW} is equivalent to $W=-\A^{-1}S$, this closure has $\kappa=1$, $\beta=1$, and $m=4$ at the response level.

\begin{proposition}\label{prop:oldenergy}
Every sufficiently smooth decaying solution of \eqref{eq:oldW}--\eqref{eq:oldmodel} satisfies
\begin{equation}\label{eq:oldenergy}
 \frac12\frac{d}{dt}\|F\|_{L^2(d\mu_5)}^2
 +\nu\|\nabla F\|_{L^2(d\mu_5)}^2
 =2\|\nabla W\|_{L^2(d\mu_5)}^2.
\end{equation}
The production term scales like $\lambda^3$ under \eqref{eq:Fscale}, whereas the viscous term scales like $\lambda$.
\end{proposition}

\begin{proof}
Since $\operatorname{div}_{\mu_5}b_4=0$, the transport term vanishes in the $L^2(d\mu_5)$ identity.  Thus
\[
 \frac12\frac{d}{dt}\|F\|_2^2+\nu\|\nabla F\|_2^2
 =2\int W_zF^2\,d\mu_5.
\]
Integration by parts in $z$ and \eqref{eq:oldW} give
\[
 \int W_zF^2\,d\mu_5
 =-\int W\,\partial_z(F^2)\,d\mu_5
 =-\int W\Delta_5W\,d\mu_5
 =\int|\nabla W|^2\,d\mu_5.
\]
This proves \eqref{eq:oldenergy}.  The scaling statement is the case $\beta=1$ of \cref{prop:scaling}.
\end{proof}

There is a simple one-dimensional diagnostic inside this modified equation.  Suppose formally that $F$ and $W$ are independent of the four transverse 5D radial coordinates and depend only on $z$.  Under decay,
\[
 W_{zz}=(F^2)_z
 \quad\Longrightarrow\quad
 W_z=F^2.
\]
Then \eqref{eq:oldmodel} becomes
\begin{equation}\label{eq:axialdiag}
 F_t+4WF_z=\nu F_{zz}+2F^3,
 \qquad W_z=F^2.
\end{equation}
For even, nonnegative data decreasing on $z>0$, one has $W\ge0$ and $F_z\le0$ on $z>0$, so
\[
 F_t-\nu F_{zz}=2F^3-4WF_z\ge2F^3.
\]
Thus the same closure that produces the positive energy law \eqref{eq:oldenergy} also contains an axial Fujita-type reaction mechanism.  The diagnostic is deliberately not promoted to a 5D singularity theorem: transverse independence has infinite 5D energy on the whole lifted space.  A localized 5D singular profile would require a separate construction and tail analysis.

The modified equation \eqref{eq:oldW}--\eqref{eq:oldmodel} is therefore a coherent response model with a sign-definite instability mechanism.  It should not be confused with the physical axisymmetric recovery: the physical source drives an independent vorticity state $G$ and uses the kinematic coefficient $m=2$, whereas \eqref{eq:oldmodel} uses instantaneous one-inversion response and $m=4$.

\section{Discussion}

The one-dimensional theorem and the 5D response identity point to the same structural issue from different directions.

First, the sign of the swirl-generated vorticity response changes the algebraic type of the Hou--Li reduction.  With the standard sign, the local quadratic form is encoded by one complex square.  With the reversed sign it factorizes over two real variables, and a large positive cone becomes a pair of scalar reaction-diffusion equations sharing an outward transport.  The viscous theorem shows that this is not an inviscid or codimension-one artifact: every nontrivial datum in the cone blows up for every positive viscosity.

Second, the domain normalization is part of the mechanism.  The periodic model carries the zero mode \eqref{eq:Ct}; the whole-line decay model does not.  This explains why an argument that integrates the vorticity equation and then discards the spatial constant can create a qualitatively different system.  The paper therefore keeps the whole-line theorem and the periodic Hou--Li theory logically separate.

Third, the 5D calculation identifies three independent response coordinates.  The coefficient $m$ decides which divergence law is enforced: $m=2$ is physical three-dimensional incompressibility, while $m=4$ is weighted 5D incompressibility.  The sign of $\kappa$ decides whether the instantaneous source response pumps or damps the $F$ energy.  The exponent $\beta$ determines the elliptic depth, with $\beta=2$ singled out by Navier--Stokes scaling.  The physical source-to-velocity chain also has temporal memory through $G$.  Replacing that chain by a scalar elliptic closure can therefore change four things at once: geometry, sign, depth, and memory.

The one-inversion model lies exactly on the pumping and supercritical side of the response family: \eqref{eq:oldenergy} gives a sign-definite energy source, while \cref{prop:scaling} shows that one elliptic inversion is supercritical relative to diffusion.  The physical response differs simultaneously in transport geometry, sign, temporal memory, and elliptic depth.

The connection with Tao's averaged model should be read at this structural level.  Tao proves that the energy identity and generic harmonic-analysis estimates do not determine regularity for the full 3D equation \cite{Tao2016}.  Here the result is much more specialized: it concerns a hierarchy of axisymmetric reductions and response closures.  Nevertheless, it shows that even after the familiar transport, stretching, viscosity, and scaling architecture has been retained, the orientation and causal depth of vorticity response can determine whether a positive singular mechanism is available.

A natural higher-dimensional problem is therefore precise: determine whether the 5D-solenoidal pumping model \eqref{eq:oldW}--\eqref{eq:oldmodel} possesses a localized finite-energy singular profile, or whether localization restores an obstruction absent from the axial diagnostic.  This is an existence-or-rigidity problem for the modified PDE itself.

\appendix

\section{Continuation inside the viscous cone}\label{app:continuation}

For completeness we record the continuation fact used in \cref{thm:main}.  The argument also explains why the theorem gives amplitude blowup rather than an unspecified loss of regularity.

\begin{proposition}\label{prop:continue}
Let $\nu>0$ and let $(a_+,a_-)$ be a smooth solution of \eqref{eq:twochannel} on $[0,T)$ satisfying the cone conclusions of \cref{lem:cone}.  If
\begin{equation}\label{eq:Astar}
 \sup_{0\le t<T}
 \bigl(\|a_+(t)\|_\infty+\|a_-(t)\|_\infty\bigr)<\infty,
\end{equation}
then the solution extends smoothly beyond $T$.
\end{proposition}

\begin{proof}
Let the bound in \eqref{eq:Astar} be $A$.  We first control the masses
\[
 M_j(t)=\int_\R a_j(x,t)\,dx.
\]
Because the channels are nonnegative, \eqref{eq:mass-id} gives
\[
 M_j'(t)
 =2\int a_j^2+\int a_ja_k
 \le (2\|a_j\|_\infty+\|a_k\|_\infty)M_j(t)
 \le 3A M_j(t).
\]
Hence
\begin{equation}\label{eq:massbound}
 M_j(t)\le M_j(0)e^{3At}.
\end{equation}
Since $V(0,t)=0$ and $V_x=a_++a_-$,
\begin{equation}\label{eq:Vbound}
 \|V(t)\|_\infty
 \le M_+(t)+M_-(t),
\end{equation}
so the drift stays bounded on $[0,T)$.

Next set $h_j=(a_j)_x$.  Equation \eqref{eq:hpm} is linear in $h_j$ once $a_\pm$ and $V$ are fixed.  The maximum principle applied to $h_j$ and $-h_j$ gives
\begin{equation}\label{eq:hbound}
 \|h_j(t)\|_\infty
 \le \|h_j(0)\|_\infty e^{2At}.
\end{equation}
We now differentiate repeatedly.  Let $D=\partial_x$ and $w_{j,n}=D^na_j$.  Applying $D^n$ to the $j$th equation gives
\begin{equation}\label{eq:higher}
 (w_{j,n})_t+V(w_{j,n})_x
 =\nu(w_{j,n})_{xx}
 +(2a_j-nV_x)w_{j,n}+P_{j,n},
\end{equation}
where $P_{j,n}$ is a finite sum of products of spatial derivatives of $a_+$ and $a_-$ of order strictly less than $n$.  This form follows because the only term containing $D^{n+1}a_j$ is $V D^{n+1}a_j$, and the only additional order-$n$ contribution from differentiating $V(a_j)_x$ is $nV_xD^na_j$.

Assume inductively that all derivatives of both channels through order $n-1$ are bounded on $[0,T)\times\R$.  Then $P_{j,n}$ is bounded, while the coefficient $2a_j-nV_x$ is bounded by a constant depending only on $n$ and $A$.  The parabolic maximum principle applied to \eqref{eq:higher}, followed by Gronwall, yields a uniform bound for $D^na_j$ on $[0,T)$.  Starting from \eqref{eq:hbound}, induction controls every spatial derivative.  The equations then control time derivatives as well.

These bounds are precisely the a priori estimates required by the standard local parabolic existence theorem to restart the solution at times tending to $T$.  Therefore the solution extends beyond $T$.
\end{proof}

\section*{Statements and Declarations}
\paragraph{Competing interests.}
The author declares no competing interests.

\paragraph{Data availability.}
No datasets were generated or analyzed during the current study.

\paragraph{Use of generative AI.}
Gemini was used during manuscript preparation for exploratory algebraic checks, literature-navigation assistance, and language editing.  The author independently verified the mathematical statements, proofs, references, and final text and assumes full responsibility for the content.

\end{document}